\documentclass{amsart}

\usepackage{booktabs}
\usepackage{hyperref}
\hypersetup{
    colorlinks,%
    citecolor=blue,%
    filecolor=blue,%
    linkcolor=blue,%
    urlcolor=blue
}
\usepackage[all]{xy}
\usepackage[alphabetic,backrefs,lite]{amsrefs}

\newtheorem{theorem}{Theorem}[section]
\newtheorem{lemma}[theorem]{Lemma}

\newtheorem{corollary}[theorem]{Corollary}
\theoremstyle{definition}

\newtheorem{definition}[theorem]{Definition}

\newtheorem{remark}[theorem]{Remark}
\theoremstyle{remark}

\numberwithin{equation}{section}

\def\cc{{\mathbb C}}
\def\kk{{\mathbb K}}

\def\zz{{\mathbb Z}}
\def\rr{{\mathbb R}}

\def\pp{{\mathbb P}}

\def\Ish{{\mathcal I}}
\def\Osh{{\mathcal O}}
\def\R{{\mathcal R}}
\def\I{\mathcal{J}_{\rm irr}}

\def\Nef{\operatorname{Nef}}
\def\Mov{\operatorname{Mov}}
\def\Cl{\operatorname{Cl}}
\def\Pic{\operatorname{Pic}}

\def\WDiv{\operatorname{WDiv}}

\def\Spec{\operatorname{Spec}}

\def\ep{\varepsilon}

\begin{document}
\title{Hypersurfaces in Mori dream spaces}
\author{Michela Artebani}
\author{Antonio Laface}
\begin{abstract}
Let $X$ be a hypersurface
of a Mori dream space $Z$.
We provide necessary
and sufficient conditions
for the Cox ring $\R(X)$
of $X$
to be isomorphic 
to $\R(Z)/(f)$, where 
$\R(Z)$ is the Cox ring 
of $Z$ and $f$ is a defining
section for $X$.
We apply our results to 
Calabi-Yau hypersurfaces 
of toric Fano fourfolds.
Our second application is to
general degree $d$ 
hypersurfaces in $\mathbb P^n$ 
containing a linear subspace 
of dimension $n-2$.
\end{abstract}
\maketitle

\section*{Introduction}
Let $Z$ be a projective 
Mori dream space, that is 
a normal projective variety 
with finitely generated 
class group and finitely 
generated Cox ring:
\[
 \mathcal R(Z)
 =
 \bigoplus_{[D]\in \Cl(Z)} 
 H^0(Z,\mathcal O_Z(D)).
\]
Given an inclusion $i: X\to Z$
of a closed irreducible 
normal subvariety $X$ of $Z$ such 
that the restriction of 
Weil divisors to $X$ is well 
defined, there is a natural 
homomorphism 
$i_{\mathcal R}:
\mathcal R(Z)\to \mathcal R(X)$. 
In order to relate the 
two Cox rings it is natural 
to ask whether such 
homomorphism is surjective 
and what is its kernel.
In~\cite{Hau} J. Hausen 
studied the case when 
$Z$ is a smooth toric 
variety.
In particular, he proved 
that  $\mathcal R(X)$ is 
isomorphic to a quotient 
of $\mathcal R(Z)$ via 
the homomorphism $i_{\mathcal R}$ 
if the inclusion 
$i:X\to Z$ is {\em neat},
which essentially means 
that the pull-back 
$i^*:\Cl(Z)\to \Cl(X)$ 
is well defined and is 
an isomorphism.
This result and its proof 
can be extended to the 
case when $Z$ is a factorial 
Mori dream space.
In this paper we use this 
generalized version of 
J. Hausen's theorem to study 
the case when $X$ is a 
hypersurface in $Z$.
More precisely, when $X$
is a normal, irreducible
and closed  
hypersurface of a 
Mori dream space $Z$,
we find necessary and 
sufficient conditions 
for the homomorphism 
$i_{\mathcal R}$ to induce
an isomorphism
$\mathcal R(Z)/(f)
\cong 
\mathcal R(X)$, where $f$ 
is a defining section for $X$. 
In case $Z$ is factorial, 
such conditions are the 
following ones: 
the pull-back $i^*:\Cl(Z)\to \Cl(X)$ 
is an isomorphism and 
the irrelevant locus 
has codimension $\geq 3$ 
in $\bar Z={\rm Spec}(\mathcal R(Z))$.
This generalizes the
main result of~\cite{Jow10}
to the non-smooth case.
Moreover, in case $X$ 
is the generic element 
of an ample and spanned 
linear series on $Z$, 
we show that the latter 
condition on the 
codimension of the irrelevant
locus is enough to
guarantee the isomorphism.

The paper is organized as
follows.
In Section 1 we introduce 
good and neat embeddings.
Section 2 contains our main
theorem~\ref{hyp} together with
its corollary about
general elements
in ample linear series 
of Mori dream spaces.
An application of Theorem~\ref{hyp}
to smooth Mori dream
Calabi-Yau hypersurfaces
of smooth toric Fano
varieties is given in Section 3.
Finally in Section 4 we 
compute the 
Cox ring of a general 
degree $d$ hypersurface 
in $\mathbb P^n$ containing 
a linear subspace of 
dimension $n-2$.

\subsubsection*{Acknowledgments}
 It is a pleasure to thank
 Elaine Herppich and
 Gian Pietro Pirola for several
 useful discussions.

\section{Embeddings in Mori dream spaces}
In what follows  $Z$  will be a 
normal projective variety over 
an algebraically closed field 
$\kk$ of characteristic zero 
with finitely generated class 
group $\Cl(Z)$. 
We briefly recall the definition 
of the Cox ring of $Z$ as given 
in~\cite[Chapter I, \S4]{ArDeHaLa}.
Let $K$ be a subgroup of $\WDiv(Z)$ 
such that the natural homorphism 
$c:K\to \Cl(Z)$, mapping $D$ to 
its class $[D]$, is surjective.
Let $K^0=\ker(c)$ and fix
a character $\chi:K^0\to \mathbb K(Z)^*$ 
with ${\rm div}(\chi(D))=D$ 
for all $D\in K^0$.
Consider the following sheaves
graded by $K$ and $\Cl(Z)$
respectively:
\[
 \mathcal S
 :=
 \bigoplus_{D\in K} 
 \Osh_Z(D),
 \quad
 \R
 =
 \mathcal S/\Ish,
 \]
where $\Ish$ is the 
ideal sheaf locally
generated by elements 
of the form $1-\chi(D)$,
with $D\in K^0$. 
The isomorphism class of 
$\mathcal R(Z)$ as a 
$\Cl(Z)$-graded ring does 
not depend on the choices 
of $K$ and $\chi$, so that 
it is called the 
{\em Cox ring} of $Z$. 
The variety $Z$ is called 
a {\em Mori dream space} 
if $\mathcal R(Z)$ is 
finitely generated. 

For any effective class 
$w\in\Cl(Z)$ we define 
the $\zz$-graded algebra 
$\R(Z,w):=\oplus_{n\in\zz}\R(Z)_{nw}$
and let $\R(Z,w)_{>0}$ be 
the subalgebra
consisting of all the 
positively graded elements.
We recall that the 
{\em irrelevant ideal} 
$\I(Z)$ 
of $\R(Z)$ is defined 
to be the radical 
of the ideal generated 
by the subalgebra 
$\R(Z,w)_{>0}$,
for any choice of an 
ample class $w$
(see~\cite[Chap. I, \S 6]{ArDeHaLa}).
Let
\[
 \bar Z:={\rm Spec}\, 
 \R(Z),
 \quad
 \hat Z=\bar Z-V(\I(Z)).
\]
The open subset $\hat{Z}$ 
is known to be {\em big} 
in $\bar{Z}$, that is its 
complementary has codimension 
at least $2$.
Moreover there exists a 
good quotient $p_Z: \hat{Z}\to Z$ 
with respect to the action of 
the quasi-torus 
$H:=\Spec(\kk[\Cl(Z)])$ induced by the 
$\Cl(Z)$-grading of the Cox ring 
(see~\cite[Chap. I, \S 6]{ArDeHaLa}).

In what follows we will 
denote by $Z_F$ the locus of 
factorial points of $Z$, that is
points $z\in Z$ such that the
local ring $\Osh_{Z,z}$ is
a unique factorization domain.
Observe that since $Z$ is normal,
then $Z_F$ is big in $Z$.
Let $X$ be a closed irreducible 
subvariety of $Z$ 
and let $i: X\to Z$ be the inclusion
map. We define $\hat{X} := p_Z^{-1}(X)$ 
and let $\bar{X}$ be the closure 
of $\hat X$ in $\bar Z$. 
We have the following 
commutative diagram:
\[
 \xymatrix{
  \bar X\ar[r]& \bar Z \\
  \hat X\ar[r]\ar[u]\ar[d]_{p_X}
  & \hat Z\ar[u]\ar[d]^{p_Z}
  & \hat{Z}_F\ar[l]\ar[d]^-{p_{Z_F}}\\
  X\ar[r]^{i} & Z 
  & Z_F\ar[l],
  }
\]
where $p_X$ and $p_{Z_F}$ are 
restrictions of $p_Z$, while all the
remaining maps are inclusions.
In order to define a pull-back 
map on Weil divisors of $Z$ we 
will need the following assumption.
\begin{definition}
The embedding $i: X\to Z$ is {\em good}
if $i^{-1}(Z_F)$ is big in $X$. 
\end{definition}
Assume that $i$ is good and 
let $D=\sum_ia_iD_i$ be a 
Weil divisor of $Z$, where 
$D_i$ are prime divisors 
and $a_i$ are positive integers.
 Let $D\cap Z_F$ be the
restriction of $D$ to $Z_F$ 
defined as $\sum_ia_i(D_i\cap Z_F)$.
Observe that $D\cap Z_F$ 
is a Cartier divisor 
since $Z_F$ is factorial.
Thus we define the pull-back 
of $D$ via $i^*$ to be:
\[
  i^*(D)
  :=
  \overline{i^*(D\cap Z_F)},
\]
where the overline denotes
the closure of the corresponding 
Cartier divisor in $X$. This closure
is unique due to the assumption that 
$i^{-1}(Z_F)$ is big in $X$.

The pull-back $i^*$ defined 
between the groups of Weil 
divisors induces a pull-back map 
between the class groups of $Z$ 
and $X$ (that will be denoted 
with the same symbol).
Such map can be obtained as follows. 
Observe that
$\Cl(Z)\cong\Cl(Z_F)\cong\Pic(Z_F)$, 
where the first isomorphism 
is due to the
fact that $Z_F$ is big in $Z$ and the
second to the fact that 
$Z_F$ is factorial.
The same holds by substituting 
$Z$ with $X$ 
and $Z_F$ with $i^{-1}(Z_F)\cap X_F$,
where $X_F$ is the factorial
locus of $X$.
This allows to define a pull-back map
$i^*:\Cl(Z)\to\Cl(X)$
by means of the following 
commutative diagram:
\[
 \xymatrix{
 \Cl(Z)\ar[d]_{\cong} \ar[r]^{i^*}&
 \Cl(X)\\
 \Pic(Z_F)\ar[r]
 & \Pic(i^{-1}(Z_F)\cap X_F)
 \ar[u]_{\cong},
 }
\]
where the lower horizontal arrow 
is induced by the usual pull-back 
of Cartier divisors, which clearly 
respects linear equivalence. 
In what follows, we define 
$D_X := i^*(D)$.

Consider now the class $w$ of
a divisor $D\in K$.
Given a non-zero element 
$f\in{\R(Z)}_w$ there exists 
a unique $\tilde f\in\mathcal S(Z)_D$ 
which is projected to $f$
via the quotient map
${\mathcal S}(Z)\to\R(Z)$
(see \cite[Chapter I, \S 3]{ArDeHaLa}).
\begin{definition}
With the same notation as above,
we say that an effective 
Weil divisor 
$E$ is {\em defined by} 
$f\in{\R(Z)}_w$
if $E={\rm div}(\tilde f)+D$.
Moreover, we will denote
by $\bar E$ the Cartier
divisor of $\bar Z$,
which is defined 
by the zero locus of 
the same $f$ thought 
as a regular function
on $\bar Z$.
\end{definition}

The following definition is just a reformulation
of~\cite[Def. 2.5]{Hau}.

\begin{definition}\label{neat} 
Let $Z$ be a Mori dream space 
and $X\subset Z$ 
be a normal, irreducible closed
subvariety. The inclusion 
$i: X\to Z$ is a 
{\em neat embedding} if
\begin{enumerate}
\item $i$ is good;
\item the pull-back 
$i^*:\Cl(Z)\to \Cl(X)$ 
is an isomorphism.
\end{enumerate}
\end{definition}

\begin{remark} In the definition of 
neat embedding given 
in~\cite[Def. 2.5]{Hau} $Z$ is 
a toric variety and 
point i) is replaced by the 
requirement that the divisors 
$D_X^k := i^*(D^k)$ are 
distinct and irreducible,
where the $D^k$ are divisors
defined by a minimal set
of generators
$\{f_1,\dots,f_k\}$ 
of the Cox ring of $Z$.
It is not hard to show
that this condition
implies that $i$ is good.
\end{remark}

The proof of the following theorem
is essentially the same as that 
of~\cite[Theorem 2.6]{Hau} 
in case the ambient variety 
$Z$ is a Mori dream space.
\begin{theorem}\label{profactor} 
Let $Z$ be a Mori dream space 
and $X\subset Z$ be a normal, 
irreducible closed subvariety. 
If $X\subset Z$ is a neat embedding and 
$Z$ is factorial, then there is 
an isomorphism of $K$-graded 
$\mathcal O_X$-algebras:
\[
  \mathcal R_X\cong 
  (p_X)_*\mathcal O_{\hat X},
\]
where $\mathcal R_X$ is any Cox sheaf 
on $X$. Moreover, $\hat X$ is normal 
and $p_X:\hat X\to X$ is a characteristic 
space for $X$.
\end{theorem}

\section{Hypersurfaces in Mori dream spaces}
We will now specialize the 
results of the previous section 
to the case when $X$ is an 
irreducible 
closed hypersurface in $Z$. 
In what follows we will 
assume the inclusion $i:X\to Z$ 
to be good, so that the 
pull-back of Weil divisors 
is well defined.
Observe that the inclusion 
$i$ induces a pull-back 
homomorphism
\[
 i_{\mathcal R}:
 \mathcal R(Z)\to
 \mathcal R(X).
\]
We recall that $Z_F$
is the factorial locus of $Z$ and $\hat Z_F=p_Z^{-1}(Z_F)$.
We will denote by 
$U_X:=Z_F\cap X$ and by
$\hat U_X:=p_X^{-1}(U_X)$. 
\begin{theorem}\label{hyp}
Let $Z$ be a Mori dream space 
and let $X$ be a normal,
irreducible closed hypersurface
of $Z$ defined by $f\in\R(Z)_w$
such that the inclusion
$i:X\to Z$ is good.
Then $i_{\mathcal R}$ induces an 
isomorphism  
$\R(Z)/(f)\cong\R(X)$
if and only if the following 
conditions hold:
\begin{enumerate}
\item $\hat X$ is big in $\bar X$,
\item $\hat U_X$  is big in $\hat X$,
\item $i^*:\Cl(Z)\to \Cl(X)$ 
is an isomorphism.
\end{enumerate}
\end{theorem}
\begin{proof} 
Assume first that the 
three conditions (i), (ii)
and (iii) hold.
Consider the following 
commutative diagram:
\[
 \xymatrix{
  \hat U_X\ar[r]\ar[d]_{p_X}
  &  \hat Z_F \ar[d]^{p_Z}\\
  U_X\ar[r]_{i} & Z_F.
  }
\]
Since $Z_F$ is factorial, 
we can apply 
Theorem~\ref{profactor} to 
the inclusion map $i$.
Observe that $i$ is neat 
by (iii), $Z_F$ is big in 
$Z$ and $U_X$ is big 
in $X$ since $i$ is good.
Thus we obtain a sheaf 
isomorphism 
$\mathcal R_{U_X}\cong 
(p_X)_*\mathcal O_{\hat U_{X}}$.
Since  $\hat U_{X}\subset \hat X$ 
and $\hat X\subset \bar X$ 
are big inclusions by 
(i) and (ii), then:
\[
 \mathcal R(X)
 =\Gamma(X,\mathcal R)
 =\Gamma(U_X,\mathcal R)
 =\Gamma(\hat U_{X},\mathcal O)
 =\Gamma(\hat X,\mathcal O)
 =\Gamma(\bar X,\mathcal O),
\]
where the first equality is 
by definition, the second 
fourth and fifth are due to
the fact that the corresponding
inclusions of subsets are big,
and finally the third equality
is due to the sheaf isomorphism
given above.
The last ring is isomorphic 
to $\R(Z)/(f)$ and the 
isomorphism is induced 
by $i_{\mathcal R}$.

Conversely, if $i_\R$ induces
an isomorphism
$\mathcal R(Z)/(f)\cong\mathcal R(X)$,
then $\mathcal R(X)$ is graded 
by $i^*\Cl(Z)$, which is thus 
isomorphic to $\Cl(X)$.
Moreover $p_X:\hat X\to X$ 
is a characteristic space, 
so that $\hat X$ is big in 
$\bar X$ and $p_X$ does not 
contract divisors 
by~\cite[Chapter I, \S6]{ArDeHaLa}.
The latter property 
and the fact that $U_X$ 
is big in $X$ imply 
that $\hat{U}_X$ 
is big in $\hat X$.
\end{proof}
\begin{remark}
Let $\{f_1,\dots,f_k\}$ be a minimal 
set of generators of $\mathcal R(Z)$ 
and $\bar D^j$'s be the zero sets 
of the $f_j$'s in $\bar Z$.
Conditions (i) and (ii) 
in Theorem \ref{hyp} 
can be given in terms 
of the divisors 
$ i^*(\bar D^j)$ of $\bar X$: 
if  such divisors  
are $\Cl(X)$-prime and distinct, 
then conditions (i) and (ii) in hold.
We recall that a Weil 
divisor on $\bar X$ 
is called $\Cl(X)${\em -prime} 
if it is a finite sum of 
distinct prime divisors 
which are 
transitively permuted by 
the action of the quasi-torus
$H_X$
(see~\cite[Chapter I, \S 4]{ArDeHaLa}).
In fact, it can be proved 
that the complement 
$V\subset\bar Z$ of all the 
intersections $\bar{D}^i\cap \bar{D}^j$,
with $i$ and $j$ distinct,
is contained in $\hat Z_F$ 
and that $V\cap\hat X$ 
is big in $\bar X$  
because of the hypotheses 
on the divisors $i^*(\bar D^j)$'s.
This implies that $V\cap \hat X$
is big in $\hat X$ and that 
$\hat X$ is big in $\bar X$, 
giving (ii) and (i) respectively.
\end{remark}

Consider now an ample
and spanned class $w\in\Cl(Z)$
and let $X$ be the effective
divisor defined by a 
general $f\in\R(X)_w$. 
Given such an $w$,
we will denote by 
$\varphi_{w}:Z\to\pp^n$
the morphism 
defined by the complete 
linear series $|w|$.

\begin{corollary}\label{hyp-sec}
Let $Z$ be a Mori 
dream space 
of dimension $\geq 3$
and let $X$ be 
a closed hypersurface 
of $Z$ defined by 
a general $f\in\R(Z)_w$, 
where $w$ is an ample 
and spanned class of $\Cl(Z)$.
If $\dim(Z)=3$ we also assume 
that $f$ is very general 
and that 
${(\varphi_{w}})_*K_Z(1)$ 
is spanned.
Then the following 
are equivalent:
\begin{enumerate}
\item $\hat X$ is big in $\bar X$.
\item $i_{\mathcal R}$ 
induces an isomorphism  
$\R(Z)/(f)\cong\R(X)$; 
\end{enumerate}
\end{corollary}
\begin{proof}
We have already seen
that (ii) implies (i).
So we now show that
(i) implies (ii).
Since $w$ is ample
and spanned, and $f$ 
is general in its 
Riemann Roch space, 
then $X$ is irreducible 
and normal by  
Bertini's first theorem and 
~\cite[Theorem 1.7.1]{BeSo}.
The genericity assumption 
on $f$ and the fact 
that $w$ is spanned 
imply that $i$ is good.
Finally, the restriction map 
$i^*: \Cl(Z)\to\Cl(X)$
is an isomorphism 
by the generalized 
Lefschetz hyperplane 
theorems~\cite[Theorem 1]{RaSr1}
and~\cite[Theorem 1]{RaSr2}.

We now show that 
condition (ii) of 
Theorem~\ref{hyp} holds. 
Recall that the factorial 
locus $Z_F$ is big 
in $Z$. Since $p_Z$ 
does not contract divisors,
then $\hat{Z}_F$ is big 
in $\hat Z$.
Consider an irreducible,
not necessarily closed, 
subvariety $B$ of $Z-Z_F$, 
which intersects $X$, and
such that all the fibers 
of $p_Z$ over $B$
have the same dimension $d$.
Then $\dim(p_Z^{-1}(B))
=\dim(B)+d\leq \dim(\hat Z)-2$. 
Since $X$ is general and 
$w$ is spanned, 
the intersection
$B\cap X$ has codimension
one in $B$.
Thus 
\[
 \dim(p_X^{-1}(B\cap X))
 =
 \dim(B\cap X)+d
 \leq
 \dim(\hat X)-2.
\]
Whence $\hat{U}_X=
p_X^{-1}(Z_F\cap X)$ 
is big in $\hat X$ and
the result follows 
from Theorem \ref{hyp}.
\end{proof}

\begin{remark}\label{smooth}
In \cite[Theorem 6]{Jow10}
Shin-Yao Jow proved that if   
$Z$ is a smooth Mori dream 
space of dimension $\geq 4$ 
such that 
\begin{equation}\label{sm2}
V(\I(Z))=\bar Z-\hat Z
\text{ has codimension }
\geq 3 \text{ in } 
\bar Z,
\end{equation}
then every smooth ample 
divisor $X\subset Z$ 
is a Mori dream space
such that,
via the restriction map 
$\Pic(Z)_{\rr}=\Pic(X)_{\rr}$, 
the nef cones of $Z$
and that of $X$ coincide
and each Mori chamber 
of $X$ is a union of  
Mori chambers of $Z$.

We observe that, 
under these hypotheses, 
condition (i) of 
Theorem~\ref{hyp} 
clearly holds, 
and condition (iii) 
is given by the classical 
Lefschetz hyperplane Theorem. 
Thus Theorem~\ref{hyp} 
states that
$i_\R$ induces an isomorphism
$\R(Z)/(f)\cong\R(X)$ 
if and only if 
$\hat X$ is big in $\bar X$.
The latter condition 
is equivalent to~\eqref{sm2}.
In fact, since the class 
of $X$ in $\Cl(Z)$ is ample, 
then the irrelevant 
locus $V(\I(Z))$ 
is contained in $\bar X$ 
and equals 
$\bar X-\hat X$, 
so that it has codimension 
$\geq 3$ in $\bar Z$ 
if and only if it has 
codimension $\geq 2$ 
in $\bar X$.
\end{remark}

\begin{remark}\label{nonnef}
The condition $X$ ample 
in $Z$ is not necessary
to have $\R(X)\cong\R(Z)/(f)$.
For example consider 
the smooth toric fourfold $Z$
whose Cox ring 
$\R(Z)\cong\cc[x_1,\dots,x_6]$
has grading matrix and
irrelevant ideal:
\[
 \left[
 \begin{matrix}
 1&1&-1&1&-1&0\\
 0&0&1&1&1&1
 \end{matrix}
 \right]
 \qquad
 \I(Z) = (x_1,x_2)\cap (x_3,x_4,x_5,x_6).
\]
If we put $w_i:=\deg(x_i)$, then
the movable and the nef cone
of $Z$ are $\Mov(Z)=\langle w_1,w_3\rangle$
and $\Nef(Z)=\langle w_1,w_4\rangle$,
see~\cite[Propositions 3.2.3 and 3.2.6]{ArDeHaLa}.
Let 
\[
 f:=x_1x_2x_5^2+x_3x_4+x_6^2
\]
be in $\Gamma(Z,\Osh_Z(D))$
and let $X:=(f)+D$ be the 
prime divisor of $Z$ defined by
$f$. By the Samuel criterion~\cite{Sa}
the quotient ring $R:=\R(Z)/(f)$ 
is factorial, by choosing the
$\zz$-grading $(1,2,2,3,1)$ on the
first five variables. If we
denote by 
\[
 \bar{X} := V(f)
 \qquad
 \hat{X} := \Bar{X}-V(\I(Z)),
\]
then the restriction
$p_{X}: \hat{X}\to X$ of
the characteristic map
$p_{Z}: \hat{Z}\to Z$
is a geometric quotient with respect
to the action of $H:=(\cc^*)^2$.
Looking at $f$ and $\I(Z)$
we see that $\hat{X}$ is big
in $\Bar{X}$. This implies that
$\hat{X}$ does not admit
invertible global regular functions.
Moreover $H$ acts freely on $\hat{Z}$,
and consequently on $\hat{X}$,
since $Z$ is smooth, so that
the action of $H$
is strongly stable in the sense
of~\cite[Definition 6.4.1]{ArDeHaLa}).
Hence the quotient
map $p:\hat{X}\to X$ is a characteristic 
space by~\cite[Theorem 6.4.3]{ArDeHaLa},
so that $R$ is isomorphic to the Cox
ring of $X$.
We conclude by observing that
the class $w_3$ of $X$ in $\Cl(Z)$
is not nef.
\end{remark}

\section{Calabi-Yau threefolds in smooth 
toric Fano varieties}
We apply results of the previous section
to the case $X\subset Z$, where $Z$ is a 
smooth toric Fano variety and $X$ is a
a smooth hyperplane section of $Z$ in the
anticanonical embedding. Thus $X$ is a
Calabi-Yau threefold.
We will often make use of 
the Magma database of smooth
toric Fano varieties.
To check any Magma calculation
contained in this paper follow
these steps:
\begin{itemize}\leftskip -5mm
\item open this page: 
\url{http://www2.udec.cl/~alaface/software/T-Fano.txt}
and copy its content into 
the online Magma calculator
located here:\\
\url{http://magma.maths.usyd.edu.au/calc};
\item paste in the same window, below
the previous text, the function occurring 
in the calculation.
\end{itemize}
The result will be the
output of the corresponding
Magma calculation done
in the paper.
All the software sessions
in this section are in 
Magma code~\cite{Magma}.

\begin{theorem}\label{cy}
Let $X$ be a smooth Calabi-Yau
threefold which is hyperplane
section of a smooth toric Fano variety $Z$.
Then $\R(X)\cong\R(Z)/(f)$ if and only
if $Z$ is one of the following:\\

\begin{center}
\begin{tabular}{c|c|c}
\toprule
$Z$ & Grading matrix
& Irrelevant ideal\\
\midrule
$\pp_{\pp^2}(\Osh\oplus\Osh\oplus\Osh(2))$ &
$\left[\begin{matrix}
  0& 0& 1& 1& 0& 1\\
  1& 1& 2& 0& 1& 2
\end{matrix}
\right]$ &
$(x_1,x_2,x_5)\cap (x_3,x_4,x_6)$ \\[12pt]
$\pp_{\pp^2}(\Osh\oplus\Osh(1)\oplus\Osh(1))$ &
$\left[\begin{matrix}
    0& 0& 1& 1& 0& 1\\
    1& 1& 0& 0& 1& 1
\end{matrix}
\right]$ &
$(x_1,x_2,x_5)\cap (x_3,x_4,x_6)$ \\[12pt]
$\pp_{\pp^2}(\Osh\oplus\Osh\oplus\Osh(1))$ &
$\left[\begin{matrix}
    0& 0& 1& 1& 0& 1\\
    1& 1& 1& 0& 1& 1
\end{matrix}
\right]$ &
$(x_1,x_2,x_5)\cap (x_3,x_4,x_6)$ \\[12pt]
$\pp^2\times\pp^2$ &
$\left[\begin{matrix}
    0& 0& 1& 1& 0& 1\\
    1& 1& 0& 0& 1& 0
\end{matrix}
\right]$ & 
$(x_1,x_2,x_5)\cap (x_3,x_4,x_6)$ \\[10pt]
$\pp^4$ &
$\left[\begin{matrix}
    1& 1& 1& 1& 1
\end{matrix}
\right]$ &
$(x_1,\dots,x_5)$ \\[5pt]
\bottomrule
\end{tabular}
\end{center}
\end{theorem}
\begin{proof}
First we show that, by looking
into the Magma database of smooth 
toric Fano varieties, there are
exactly five such varieties whose 
Cox ring admits an irrelevant ideal
of codimension at least $3$.

\begin{verbatim}
> [n : n in [24..147] | 
  Length(FanoX(n))-Dimension(IrrelevantIdeal(FanoX(n))) ge 3];
[ 44, 70, 141, 146, 147 ]
\end{verbatim}

Each of these varieties, let us say $Z$,
satisfies the hypothesis of 
Theorem~\ref{hyp} by Remark~\ref{smooth},
since
$-K_Z$ is very ample and in particular
it is ample and spanned.
Hence any smooth element $X$ of 
the linear series $|-K_Z|$
has Cox ring isomorphic to 
$\R(Z)/(f)$, where $f\in H^0(Z,K_Z)$
is a defining section for $X$
in $Z$.

To obtain the grading matrix 
for the Cox ring of $Z$ and the
irrelevant ideal we again ask
Magma to do it.
For example the variety n. 44 is:
\begin{verbatim}
> FanoX(44);
Toric variety of dimension 4
Variables: $.1, $.2, $.3, $.4, $.5, $.6
The components of the irrelevant ideal are:
    ($.6, $.4, $.3), ($.5, $.2, $.1)
The 2 gradings are:
    0, 0, 1, 1, 0, 1,
    1, 1, 2, 0, 1, 2
\end{verbatim}

This gives the second and third
column of the central table of 
our theorem. To provide the 
projective models for the five 
varieties we calculate the value 
of $-K_Z^4$, which is just the
degree of $Z$ in the anticanonical 
embedding:
\begin{verbatim}
> Degree(-CanonicalDivisor(FanoX(44)));       
594
\end{verbatim}
and determine all the linear relations
within the vertices of the 
polytope defined by $-K_Z$:
\begin{verbatim}
> Kernel(Matrix(Vertices(FanoP(44))));
RSpace of degree 6, dimension 2 over Integer Ring
Echelonized basis:
( 1  1  0 -2  1  0)
( 0  0  1  1  0  1)
\end{verbatim}
Looking for varieties with these
invariants in~\cite[Proposition 3.1.1 
and pag. 1046]{Ba} we obtain the
left hand side column of our table.
\end{proof}

\begin{remark}
Observe that 
$\pp_{\pp^2}(\Osh\oplus\Osh\oplus\Osh(2))$
is isomorphic to the blow-up, along
the vertex, of the quadratic cone of 
dimension four with vertex a line,
while
$\pp_{\pp^2}(\Osh\oplus\Osh\oplus\Osh(1))$
is isomorphic to the blow-up of 
$\pp^4$ along a line.
\end{remark}

\begin{theorem}
Let $Z$ be one of the five varieties
of Theorem~\ref{cy} and let $P_Z$
be the polytope defined by $-K_Z$.
Let $Z^*$ be the toric variety 
constructed from the dual polytope
$P_Z^*$. Then the general element
of $|-K_{Z^*}|$ is a Mori dream 
Calabi-Yau variety whose Cox ring
admits just one relation: 
$\R(X^*)\cong\R(Z^*)/(f^*)$.
\end{theorem}
\begin{proof}
To prove the theorem it is enough
to show, by Corollary~\ref{hyp-sec},
that the irrelevant ideal
of the Cox ring $\R(Z^*)$ has codimension
at least $3$. The following Magma
command calculates the codimension
of the irrelevant ideal for 
any such $Z^*$:
\begin{verbatim}
> [Length(FanoDualX(n))-Dimension(IrrelevantIdeal(FanoDualX(n))): 
n in [44,70,141,146,147]];
[ 3, 3, 3, 3, 5 ]
\end{verbatim}
\end{proof}

\section{Hypersurfaces of $\pp^n$
containing a codimension two linear space}
Denote by $X_d$ a general
degree $d$ smooth hypersurface of $\pp^n$
containing a linear subspace $L$ 
of codimension $2$. Here we assume
$X_d$ to be at least three dimensional.
An elementary calculation shows
that $X_d$ is singular at a finite
set of points lying on $L$.
However, after blowing up $L$ in 
$\pp^n$, the resulting strict transform
$X$ of $X_d$ is smooth.
Denote by $\pi: Z\to \pp^n$ 
the blow-up map at $L$.
The Cox ring of $Z$ is a polynomial
ring in $n+2$ variables with
grading matrix
\[
 \left[\begin{matrix}
  1& 1& 1& \dots & 1 & 0\\
  -1& -1& 0&\dots & 0 & 1
 \end{matrix}
 \right]
\]
The last variable, $x_{n+2}$, 
corresponds to the exceptional 
divisor of the blow-up.
Denote by $\bar{Z}:=\Spec(\R(Z))$ and by 
$\hat{Z}\subseteq\bar{Z}$ the characteristic space
of $Z$ together with the characteristic
map $p:\hat{Z}\to Z$. Observe that
the irrelevant ideal of $Z$ is
\[
 \I(Z)= (x_1,x_2)\cap (x_3,\dots,x_{n+2}).
\]
Let $\hat{X}:=p^{-1}(X)$
and $\bar{X}$ be its Zariski closure of
$\bar{Z}$. The equation of $\bar{X}$ in 
$\bar{Z}$ is:
\[
  x_1f+x_2g = 0,
\]
where $f$ and $g$ are general (very general
if $n=3$) polynomials of degree $(d-1)e_1$.
Observe that the quotient ring $\R(Z)/(x_1f+x_2g)$
is non-factorial, since, $\bar{x}_1$ 
is irreducible in the quotient ring, but 
it does not divide $\bar{g}$ due
to the generality assumption on $X_d$.
We introduce a new variable $x_0$ in order
to obtain factoriality. Consider the $\zz^2$-graded
ring
\[
  R := \cc[x_0,\dots,x_{n+2}]/(x_0x_2-f,x_0x_1+g),
\]
where the gradings of $x_i$, with $i=1,\dots,n+2$,
are given before and $\deg(x_0)=(d-2)e_1+e_2=
 \left[\begin{matrix}
  d-2\\
  1
 \end{matrix}
 \right]$


\begin{theorem}\label{Cox-4}
The Cox ring $\R(X)$ is isomorphic to $R$.
\end{theorem}
\begin{proof}
Let $Z_1$ be the toric variety whose Cox ring
$\R(Z)$ is isomorphic to the $\zz^2$-graded 
polinomial ring $\cc[x_0,\dots,x_{n+2}]$
such that $w:=e_1$ is an ample class.
The Cox ring of $Z_1$ is a polynomial
ring in $n+3$ variables with
grading matrix
\begin{equation}\label{gr-mat}
 \left[\begin{matrix}
  d-2 & 1& 1& 1& \dots & 1 & 0\\
  1 & -1& -1& 0&\dots & 0 & 1
 \end{matrix}
 \right]
\end{equation}
The first variable is the $x_0$ just
introduced, while the remaining
variables have the same grading
of those of $\R(Z)$.
The irrelevant ideal of $\R(Z_1)$ is 
\[
  \I(Z_1) = (x_0,x_3,\dots,x_{n+2})
  \cap (x_1,\dots,x_{n+1}),
\]
so that the corresponding irrelevant locus 
$\bar{Z}_1^0$ has codimension $n+1$ in $\bar{Z}_1$.
Observe that a very general element of the linear
series $|(d-1)e_1|$ in $Z_1$ can be written
in the form $\alpha(x_0x_2-f)+\beta(x_0x_1+g)$,
with $f$ and $g$ very general. Let
$Y_1\in |(d-1)e_1|$ be such a very general
element. We observe that $Y_1$ is normal, 
by~\cite[Theorem 1.7.1]{BeSo}. 
Moreover the inclusion map 
$i: Y_1\to Z_1$ induces an isomorphism 
$i^*:\Cl(Z_1)\to\Cl(Y_1)$, 
by~\cite[Theorem 1]{RaSr1}.
Let $Y_2\subseteq Y_1$ be a very general
element of the linear series $|(d-1)e_1|$
cut out in $Z_1$ by the equations
$x_0x_2-f=0,x_0x_1+g=0$.
As before we see that $Y_2$ is
normal and the inclusion map 
$j: Y_2\to Z_2$ induces an isomorphism 
$j^*:\Cl(Y_1)\to\Cl(Y_2)$.
Let $\hat{Y}_i := p^{-1}(Y_i)$ and
let $\bar{Y}_i$ be the Zariski closure
of $\hat{Y}_i$ in $\bar{Z}_i$.
Observe that the intersection 
$\{x_k=0\}\cap \bar{Y}_i$ is irreducible 
for any $i$ and $k$. Moreover as $k$ varies,
the intersections are distinct.
Thus, since $\deg(x_0x_2-f)=\deg(x_0x_1+g)=(d-1)e_1$ 
is ample and spanned in both $Z_1$ and $Y_1$,
then $\hat{Y}_i$ is a characteristic space
for $Y_i$, by Theorem~\ref{hyp-sec}.
Consider now the commutative diagram:
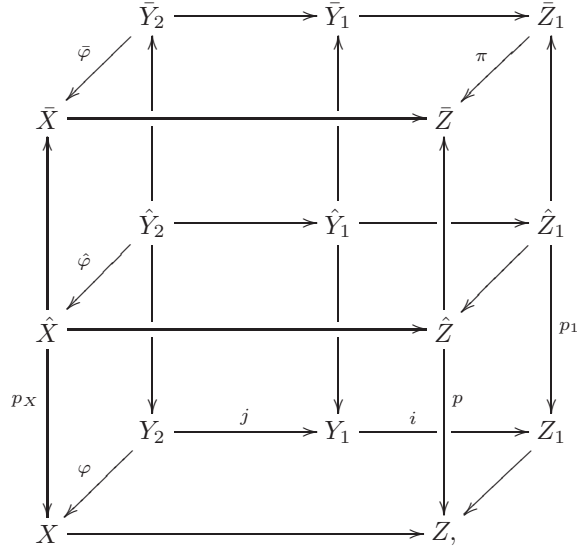
\begin{equation*} 
\xymatrix{
&\bar{Y}_2\ar[rr]\ar[dl]_-{\bar{\varphi}}
&&\bar{Y}_1 \ar[rr]
&& \bar{Z}_1\ar[dl]_-{\pi} \\ 
\bar{X}\ar[rrrr]&&&&\bar{Z}\\ 
&\hat{Y}_2\ar[rr]\ar[dl]_-{\hat{\varphi}}
\ar[dd]|!{[d];[d]}\hole \ar[uu]|!{[u];[u]}\hole
&&\hat{Y}_1 \ar[rr]|!{[r];[r]}\hole 
\ar[dd]|!{[d];[d]}\hole \ar[uu]|!{[u];[u]}\hole
&&\hat{Z}_1\ar[dl]\ar[dd]^{p_1}
\ar[uu]\\ 
\hat{X}\ar[rrrr]\ar[dd]_(.35){p_X}\ar[uu]
&&&&\hat{Z}\ar[dd]^(.35){p}\ar[uu]\\ 
&Y_2\ar[rr]^-{j}\ar[dl]_-{\varphi}
&&Y_1 \ar[rr]^(.35){i}|!{[r];[r]}\hole
&&Z_1 \ar[dl]\\ 
X\ar[rrrr]&&&& Z,
}
\end{equation*}
where all the horizontal arrows are inclusion 
maps and $\pi$ is the projection on the last 
$n+2$ coordinates and $\bar{\varphi}$ is the 
restriction of $\pi$ since $(x_0x_2-f,x_0x_1+g)
\cap\cc[x_1,\dots,x_{n+2}] = (x_1f+x_2g)$.
If $x\in\hat{Y}_2-\{x_1=x_2=0\}$, then $x_0$ 
is uniquely determined by the equations of $Y_2$, 
so that $\hat{\varphi}$ is one to one on this
big open subset of $\hat{Y}_2$.
Observe that $\{x_1=x_2=0\}$ is not contained
in $\hat{X}_d$ since it is contained in the
irrelevant locus of $\bar{Z}$.
Thus $\varphi$ is an isomorphism in codimension 1
so that $\R(X)\cong\R(Y_2)\cong\Osh(\bar{Y}_2)=R$.

\end{proof}

\begin{remark}
Unfortunately Corollary~\ref{hyp-sec}
can not be applied to 
compute $\mathcal R(X)$ 
in the case $n=3$.
In fact, to show that the 
generalized Lefschetz 
Theorem still holds for 
$Y_2 \subset Y_1$ one needs 
the extra condition that
$\varphi_*(K_{Y_1})(1)$ is globally 
generated, where $\varphi$ 
is the morphism on $Y_1$ 
associated to the complete 
linear series $|Y_2|$.
According to the grading 
matrix~\eqref{gr-mat} 
the class $D$ of $Y_1$ 
has degree $(d-1)e_1$ 
and the zero locus of 
all monomials in the 
Riemann-Roch space of $D$ 
coincides with the 
irrelevant locus of 
$\bar Y_2$. 
Thus $D$ is very ample, 
so that the previous condition 
is equivalent to the 
base-point freeness of 
the divisor $K_{Y_2}+D$.
However, since $K_{Y_2}+D=K_{Y_1}+2D$ 
has degree $(d-4)e_1$ 
then the common zero 
locus of its monomials 
is the union of the two 
components $V(x_1,x_2,x_3,x_4)$ 
and $V(x_3,x_4,x_5)$.
The first component 
is contained into the 
irrelevant locus, 
while the second one 
intersects $\bar Y_2=V(x_0x_2-f)$ 
along a two-dimensional 
orbit which, under 
the action of the torus, 
becomes a point $p$ of $Y_2$. 
Hence the base locus of 
$K_{Y_1}+2D$ consist 
exactly of $p$, so that 
the required condition fails.
In the following subsection 
we will prove the analogous 
of Theorem~\ref{Cox-4} 
in the 3-dimensional case 
using a different technique.
\end{remark}

\subsection{Calculating the Cox 
ring when $n=3$}
We now proceed to calculate 
a presentation for the Cox ring
in the three dimensional case.

\begin{lemma}\label{h1}
Let $a$, $b$ be positive integers
such that $a+(2-d)b>0$ and 
$a+b>d-4$. Then $D=aH+bL$ is
non-special, that is $h^1(D)=0$.
\end{lemma}
\begin{proof}
By the adjunction formula 
the canonical
divisor $K_X$ is linearly 
equivalent to $(d-4)H$.
Write $D-K_X=N+\ep L$, where
$N:=(a-d+4)H+(b-\ep) L$
and $\ep:=(d-4)/(d-2)$.
Intersecting with the generators
of the extremal rays of the 
effective cone gives
\[
 N\cdot L = (a-d+4)+(b-\ep)(2-d)
 > 0
 \qquad
 N\cdot (H-L) = a+b-(d-4)-\ep > 0,
\]
where the second inequality
follows since $0<\ep<1$.
Thus $N$ lies in the interior of
the nef cone, so that it is ample.
Thus we conclude by the Kawamata-Viehweg
vanishing theorem.
\end{proof}

\begin{theorem}\label{cox}
The Cox ring of $X$ is
the following $\zz^2$-graded ring\\
\begin{center}
\begin{tabular}{c|c}
$\R(X)$ & Grading matrix\\
\midrule

 $\dfrac{\cc[x_0,x_1,x_2,x_3,x_4,x_5]}
  {(x_0x_1-f, x_0x_2-g)}$
  &
 $\left[\begin{matrix}
  d-2 & 1 &1 & 1 & 1 & 0 \\
    1 &-1 &-1 & 0 & 0 & 1
  \end{matrix}\right]$\\[15pt]
\end{tabular}
\end{center}
where $f$ and $g$ are very general
polynomials in $x_1x_5, x_2x_5, x_3,x_4$. 
\end{theorem}
\begin{proof}
The grading matrix in the statement 
of the theorem is given with respect
to the classes of $H$ and $L$.
We now choose $s_0,\dots,s_5$
non-zero sections of the Cox ring
such that the degree of $s_i$ 
is the $i+1$-th column of the 
grading matrix.
Let $s_0\in H^0((d-2)H+L)$
be a section which does 
not vanish on $L$. 
There is such a section
by Lemma~\ref{h1}
and Kawamata-Viehweg.
Let $s_1,s_2$ be a basis 
of $H^0(H-L)$ and let
$s_3,s_4$ be such that 
$H^0(H)=\langle s_2s_5, s_3s_5, 
s_3, s_4\rangle$.
Finally let $s_5$ be a section 
defining $L$.

To prove that the $s_i$
actually generate the Cox ring
it is enough to show it 
for nef divisors, since if $D$
is not nef, then $D\sim N+aL$
for some positive integer $a$ 
and some nef divisor $N$.
Thus $H^0(D)=s_5^a\cdot H^0(N)$.
Consider now
a nef divisor $D=aH+bL$.
The
exact sequence
$0\to
\Osh_X(D-2F)\to
\Osh_X(D-F)\oplus \Osh_X(D-F)
\to\Osh_X(D)\to 0
$
induces the following
exact sequence in cohomology:
\[
  \xymatrix{
  H^0(D-F)\oplus H^0(D-F)\ar[r]^-{f} &
  H^0(D)\ar[r] & H^1(D-2F),
  }
\]
where $f(u,v)=us_1+vs_2$.
Observe that since $D$ is nef
then the following hold:
\[
 a+b\geq 0
 \qquad
 k:=D\cdot L=a+(2-d)b\geq 0.
\]
We consider five cases.
\begin{enumerate}\leftskip -5mm
\item
If $b=0$, then $D=aH$. Since 
$h^1(\Osh_{\pp^3}(n))=0$ for all 
$n\in \mathbb Z$ 
by~\cite[Theorem 5.1.]{Ha},
then the restriction map 
$H^0(\Osh_{\pp^3}(a))\to H^0(aH)$ 
is surjective 
for all $a\geq 0$. 
This implies that
$H^0(aH)$ is a polynomial
in $s_1,\dots,s_5$,
for any non negative $a$.
  
\item
If $a+b=0$, then $D$
is a multiple of $F$. 
Hence the complete linear
series $|D|$ is composed 
with the pencil $|F|$
or equivalently
any element of $H^0(D)$
is a polynomial in
$s_1,s_2$.

\item 
If $k\geq 2d-3$ and $a+b>d-4$,
then $h^1(D-2F)=0$, 
by Lemma \ref{h1}. 
In this case, the map $f$ 
is surjective, so that the 
elements of $H^0(D)$ 
are polynomials in $s_1,s_2$ 
and sections 
in $H^0(D-F)$.

\item
If $k\leq 2d-4$ and
$b\geq 1$,
consider the commutative diagram:
\[
  \xymatrix{
  & & H^0(kH)\ar[d]^{\cdot s_0^b}
  \ar[rd]^-{r'}& \\
  0\ar[r] & 
  H^0(D-L)\ar[r]^-{\cdot s_5} & 
  H^0(D)\ar[r]^-{r} & 
  H^0(\Osh_{\pp^1}(k))\ar[r] & 
  H^1(D-L),
  }
\]
where the bottom row is exact,
the vertical map is multiplication
by $s_0^b$ since $D=kH+b((d-2)H+L)$
and $r'$ is the 
restriction map to $L$.
Since $a\geq (d-2)b$ and 
$a+b\geq (d-1)b\geq d-1$,
then $h^1(D-L)=0$ 
by Lemma~\ref{h1}, so that 
$r$ is surjective.
Moreover $r'$ is surjective
for any $k\geq 0$ since the restrictions of $s_3,s_4$
to $L$ span $H^0(\Osh_{\pp^1}(1))$.
Thus any element
$s\in H^0(D)$ 
is a sum $s=u\,s_5+v\,s_0^b$,
with
$u\in H^0(D-L)$ and 
$v\in H^0(kH)$.

\item
If $a+b\leq d-4$, then 
$b\leq 0$ since $a\geq (d-2)b$.
Let $C\in |H-L|$ be the curve 
defined by $s_1$.
Observe that $C$ is a 
plane curve of 
degree $d:=a+b$.
Thus we have a commutative 
diagram:
\[
\xymatrix{
 & & H^0(\Osh_{\pp^3}(d))\ar[r]^-{r_1}
 \ar[d]_{\gamma}
 & H^0(\Osh_{\pp^2}(d))\ar[d]^-{r_2}\\
 0\ar[r] & 
 H^0(D-F)\ar[r]^-{\cdot s_1} & 
 H^0(D)\ar[r]_-r & H^0(D_{|C}),
}
\]
where $r_1$, $r_2$, $r$ are 
restriction maps, $\gamma$ is the restriction 
to $X$ composed with the 
multiplication by $s_2^{-b}$
and the bottom row is
exact.
Since $h^1(\Osh_{\pp^2}(h))=0$, 
for any $h\in\zz$,
by~\cite[Theorem 5.1.]{Ha}, 
then $r_2$ is surjective.
Since both $r_1$ and $r_2$ 
are surjective, then $r$ 
is surjective as well.
Thus, any section of 
$s\in H^0(D)$ is a sum
$s=u\,s_1+v\,s_2^{-b}$,
where $u\in H^0(D-F)$
and $v\in H^0(dH)$.
\end{enumerate}
 
The previous arguments show 
that the Cox ring $\R(X)$ 
is generated by 
$s_0,\dots, s_5$.
Let $I_X$ be the kernel 
of the ring homomorphism
\[
  \cc[x_0,\dots,x_5]\to \R(X)
  \qquad
  x_i\mapsto s_i.
\]
Since the vector space 
$H^0(4H)$ is generated 
by degree $4$ polynomials 
in $s_1s_5$, $s_2s_5$, 
$s_3,s_4$ and 
$s_0s_1, s_0s_2$ 
belong to this space,
then $I_X$ contains two 
polynomials of type 
$x_0x_1-f, x_0x_2-g$ 
as in the statement.
The quotient ring 
$R=\cc[x_1,\dots,x_6]/(x_0x_1-f, x_0x_2-g)$ 
is an integral domain 
and the surjectivity of 
$R\to \R(X)$ implies that the 
corresponding morphism
$\Spec(\R(X))\to\Spec(R)$ 
is injective. Since both 
$\Spec(\R(X))$ and $\Spec(R)$ 
are four dimensional integral 
affine varieties, then we 
deduce $R = \R(X)$.
\end{proof}


\begin{bibdiv}
\begin{biblist}


\bib{ArDeHaLa}{book}{
    AUTHOR = {Arzhantsev, Ivan},
    AUTHOR = {Derenthal, Ulric},
    AUTHOR = {Hausen, Juergen},
    AUTHOR = {Laface, Antonio},
     TITLE = {Cox rings},
       URL = {http://www.mathematik.uni-tuebingen.de/~hausen/CoxRings/download.php?name=coxrings.pdf},
}

\bib{Ba}{article}{
    AUTHOR = {Batyrev, Victor V.},
     TITLE = {On the classification of toric {F}ano {$4$}-folds},
      NOTE = {Algebraic geometry, 9},
   JOURNAL = {J. Math. Sci. (New York)},
  FJOURNAL = {Journal of Mathematical Sciences (New York)},
    VOLUME = {94},
      YEAR = {1999},
    NUMBER = {1},
     PAGES = {1021--1050},
      ISSN = {1072-3374},
     CODEN = {JMTSEW},
   MRCLASS = {14M25 (14J45)},
  MRNUMBER = {1703904 (2000e:14088)},
MRREVIEWER = {Jaros{\l}aw A. Wi{\'s}niewski},
}

\bib{BeSo}{book}{
    AUTHOR = {Beltrametti, Mauro},
    AUTHOR = {Sommese, Andrew J.},
     TITLE = {The adjunction theory of complex projective varieties},
    SERIES = {de Gruyter Expositions in Mathematics},
    VOLUME = {16},
 PUBLISHER = {Walter de Gruyter \& Co.},
   ADDRESS = {Berlin},
      YEAR = {1995},
     PAGES = {xxii+398},
      ISBN = {3-11-014355-0},
   MRCLASS = {14C20 (14-02 14E35 14N05)},
  MRNUMBER = {1318687 (96f:14004)},
MRREVIEWER = {Jaros{\l}aw A. Wi{\'s}niewski},
}


\bib{Ha}{book}{
    AUTHOR = {Hartshorne, Robin},
     TITLE = {Algebraic geometry},
      NOTE = {Graduate Texts in Mathematics, No. 52},
 PUBLISHER = {Springer-Verlag},
   ADDRESS = {New York},
      YEAR = {1977},
     PAGES = {xvi+496},
      ISBN = {0-387-90244-9},
   MRCLASS = {14-01},
  MRNUMBER = {0463157 (57 \#3116)},
MRREVIEWER = {Robert Speiser},
}

\bib{Hau}{article}{
    AUTHOR = {Hausen, J{\"u}rgen},
     TITLE = {Cox rings and combinatorics. {II}},
   JOURNAL = {Mosc. Math. J.},
  FJOURNAL = {Moscow Mathematical Journal},
    VOLUME = {8},
      YEAR = {2008},
    NUMBER = {4},
     PAGES = {711--757, 847},
      ISSN = {1609-3321},
   MRCLASS = {14C20 (14J25 14L24 14M25)},
  MRNUMBER = {2499353 (2010b:14011)},
MRREVIEWER = {Ivan V. Arzhantsev},
}

 

\bib{Jow10}{article}{
    AUTHOR = {Jow, Shin-Yao}
TITLE = {A {L}efschetz hyperplane theorem for {M}ori dream spaces},
Journal = {Mathematische Zeitschrift},
volume = {264},
year = {2010},
doi = {10.1007/s00209-010-0666-9},
masid = {12566326}
}


\bib{Magma}{article}{
    AUTHOR = {Bosma, Wieb},
    AUTHOR = {Cannon, John},
    AUTHOR = {Playoust, Catherine},
     TITLE = {The {M}agma algebra system. {I}. {T}he user language},
      NOTE = {Computational algebra and number theory (London, 1993)},
   JOURNAL = {J. Symbolic Comput.},
  FJOURNAL = {Journal of Symbolic Computation},
    VOLUME = {24},
      YEAR = {1997},
    NUMBER = {3-4},
     PAGES = {235--265},
      ISSN = {0747-7171},
   MRCLASS = {68Q40},
  MRNUMBER = {MR1484478},
       DOI = {10.1006/jsco.1996.0125},
       URL = {http://dx.doi.org/10.1006/jsco.1996.0125},
}


\bib{RaSr1}{article}{
    AUTHOR = {Ravindra, G. V.}
    AUTHOR = {Srinivas, V.},
     TITLE = {The {G}rothendieck-{L}efschetz theorem for normal projective
              varieties},
   JOURNAL = {J. Algebraic Geom.},
  FJOURNAL = {Journal of Algebraic Geometry},
    VOLUME = {15},
      YEAR = {2006},
    NUMBER = {3},
     PAGES = {563--590},
      ISSN = {1056-3911},
   MRCLASS = {14C20 (14C22 14C30)},
  MRNUMBER = {2219849 (2006m:14008)},
MRREVIEWER = {Barry H. Dayton},
}

\bib{RaSr2}{article}{
    AUTHOR = {Ravindra, G. V.}
    AUTHOR = {Srinivas, V.},
     TITLE = {The {N}oether-{L}efschetz theorem for the divisor class group},
   JOURNAL = {J. Algebra},
  FJOURNAL = {Journal of Algebra},
    VOLUME = {322},
      YEAR = {2009},
    NUMBER = {9},
     PAGES = {3373--3391},
      ISSN = {0021-8693},
     CODEN = {JALGA4},
   MRCLASS = {13C20 (14C20)},
  MRNUMBER = {2567426},
MRREVIEWER = {Abdeslam Mimouni},
       DOI = {10.1016/j.jalgebra.2008.09.003},
       URL = {http://dx.doi.org/10.1016/j.jalgebra.2008.09.003},
}


\bib{Sa}{article}{
    AUTHOR = {Samuel, P.},
     TITLE = {Lectures on unique factorization domains},
    SERIES = {Notes by M. Pavman Murthy. Tata Institute of Fundamental
              Research Lectures on Mathematics, No. 30},
 PUBLISHER = {Tata Institute of Fundamental Research},
   ADDRESS = {Bombay},
      YEAR = {1964},
     PAGES = {ii+84+iii},
   MRCLASS = {13.00},
  MRNUMBER = {0214579 (35 \#5428)},
MRREVIEWER = {M. Nagata},
}

\end{biblist}
\end{bibdiv}
\end{document}